\documentclass[a4paper]{article}
\usepackage[english]{babel}
\usepackage[latin1]{inputenc}
\usepackage{amsmath,units}
\usepackage{amsthm,amsfonts,amssymb}
\usepackage{mathtools}
\usepackage{url}
\usepackage{hyperref}
\usepackage[all]{xy}
\usepackage{authblk}
\hypersetup{pdfpagemode=UseNone, pdfstartview={FitH}}

\newcommand{\lag}{\left \langle}
\newcommand{\rog}{\right \rangle}
\newcommand{\df}{\mathrel{\mathop:}=}

\usepackage{cleveref}
\newtheorem{theorem}{Theorem}
\newtheorem{remark}{Remark}

\newtheorem{definition}{Definition}
\newtheorem{proposition}{Proposition}
\newtheorem{lemma}{Lemma} 
\newtheorem{corollary}{Corollary}
\crefname{problem}{Problem}{Problems}
\Crefname{problem}{Problem}{Problems}

\begin{document}
\title{A note on drastic product logic}
\author{Stefano Aguzzoli, Matteo Bianchi}
\affil{Department of Computer Science, Università degli Studi di Milano, Via Comelico 39/41, 20135, Milano, Italy \\\url{aguzzoli@di.unimi.it}, \url{matteo.bianchi@unimi.it}}
\author{Diego Valota}
\affil{Dipartimento di Scienze Teoriche e Applicate, Università degli Studi dell'Insubria, Via Mazzini 5, 21100, Varese, Italy \\\url{diego.valota@uninsubria.it}}
\date{}
\maketitle
%\vspace{-0.7em}
\begin{abstract}
The drastic product $*_D$ is known to be the smallest $t$-norm, since $x *_D y = 0$ whenever $x, y < 1$. This $t$-norm is not left-continuous, and hence it
does not admit a residuum. So, there are no drastic product $t$-norm based many-valued logics, in the sense of \cite{eg}.
However, if we renounce standard completeness, we can study the logic whose semantics is provided by those MTL chains whose monoidal operation is the drastic product.
This logic is called ${\rm S}_{3}{\rm MTL}$ in \cite{nog}.
In this note we justify the study of this logic, which we rechristen DP (for drastic product), by means of some interesting properties
relating DP and its algebraic semantics to a weakened law of excluded middle, to the $\Delta$ projection operator and to discriminator varieties.
We shall show that the category of finite DP-algebras is dually equivalent to a category whose objects are multisets of finite chains.
This duality allows us to classify all axiomatic extensions of DP, and to compute the free finitely generated DP-algebras.
\end{abstract}
\section{Introduction and motivations}

The {\em drastic product} $t$-norm $*_D \colon [0,1]^2 \to [0,1]$ is defined as follows: $x*_D y=0$ if $x,y<1$,  $x*_D y=\min\{x,y\}$ otherwise
(see \Cref{fig:1}).
It is clear from the definition that $*_D$ is the smallest $t$-norm, in the sense that for any $t$-norm $*$ and for each $x,y \in [0,1]$
it holds that $x *_D y \leq x * y$. For this reason it is considered one of the fundamental $t$-norms (see, {\em e.g.} \cite{kmp}).
This notwithstanding, there is no drastic product $t$-norm-based logic, in the sense of \cite{eg}, since $*_D$ is not left-continuous,
and hence it has no associated residuum.

In \cite{ssac} Schweizer and Sklar introduce a class of $t$-norms which arise as modifications of the drastic product $t$-norm in such a way
to render them border continuous. In this paper the authors explicitly state
``{\em The result is a $t$-norm which coincides with [the drastic product] over most of the unit square}''.
In \cite{jen}, Jenei introduced left-continuous versions of the above mentioned $t$-norms, which he called {\em revised drastic product} $t$-norms,
as an example of an ordinal sum of triangular subnorms, namely the
ordinal sum of the subnorm which is constantly $0$ with the $t$-norm $\min\{x,y\}$.  The logic RDP based on these $t$-norms has been studied by Wang in \cite{rdp},
where,
by way of motivation, the author recalls the argument of \cite{ssac} about RDP $t$-norms as good approximators of the drastic product.
As RDP is a prominent extension of the logic of weak nilpotent minimum WNM, in \cite{bval} Bova and Valota introduce a categorical duality
for finite RDP-algebras, as a step towards a duality for the case of WNM-algebras.

As it has already been pointed out for  \cite{ssac} and \cite{rdp},
one justification held for the study of RDP is that revised drastic product $t$-norms make good approximations of $*_D$, in the sense that the graph of such
a $t$-norm can be chosen to coincide with $*_D$ up to a subset of $[0,1]^2$ of euclidean measure as small as desired.

A simple observation will show that RDP $t$-norms are as good approximators of $*_D$ as $t$-norms isomorphic (as ordered commutative semigroups) with
{\L}ukasiewicz $t$-norm. Consider, for instance, the parameterised family of $t$-norms introduced by the same Schweizer and Sklar in \cite{ssac},
defined as follows (here we consider only positive real values for the parameter $\lambda$): $x *_{SS}^\lambda y := \max\{0,x^\lambda + y^\lambda -1\}^{1/\lambda}$.
These $t$-norms, being continuous and nilpotent, are all isomorphic to {\L}ukasiewicz $t$-norm, which is obtained by choosing $\lambda = 1$.
It is easy to verify that, for each $c \in (0,1)$, the unique RDP $t$-norm $*_c$ having $c$ as negation fixpoint (that is, $\sim c = c$), has its zeroset
$\{(x,y) \in [0,1]^2 \mid x *_c y = 0\}$ properly included in the zeroset of $*_{SS}^\lambda$ for each $\lambda \geq \log_{1/c} 2$. See Fig. \ref{fig:1} for an example.
\begin{figure}\label{fig:1}
\vspace{-1em}
\begin{center}
\includegraphics[width=3.5cm]{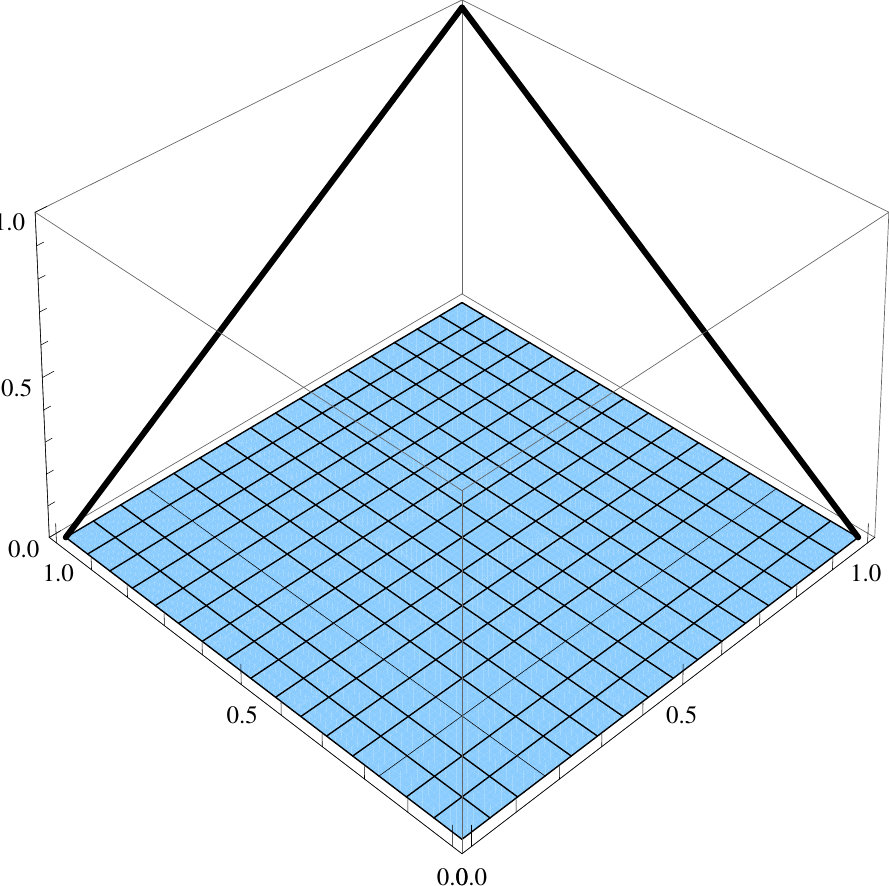}
\qquad
\includegraphics[width=3.5cm]{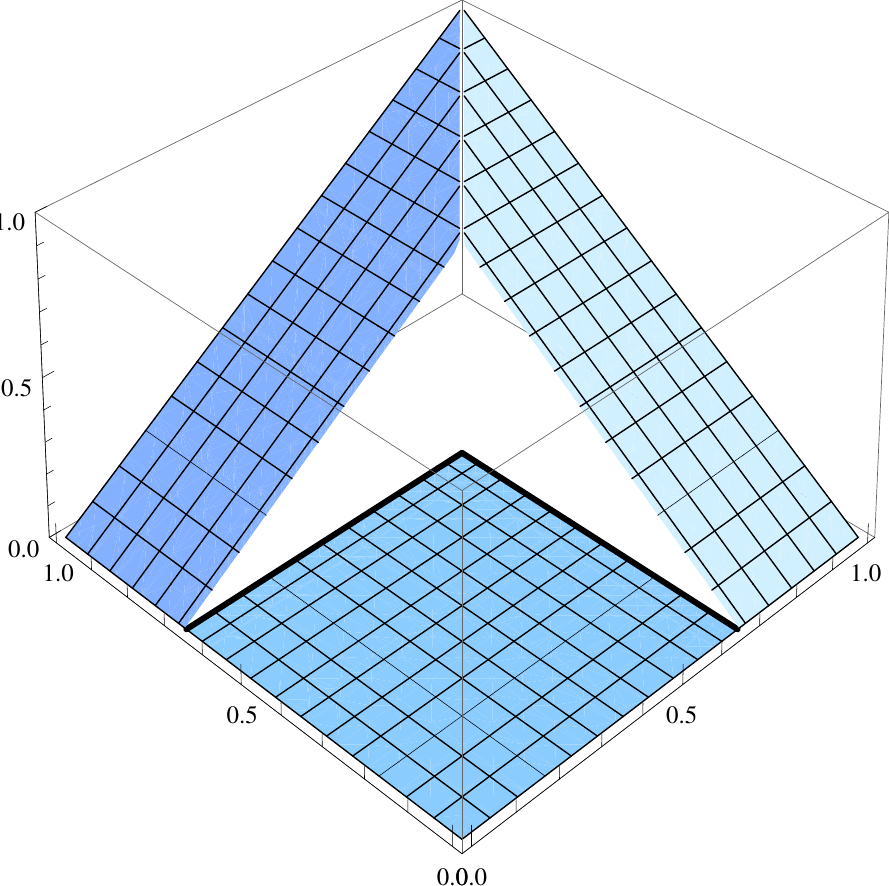}
\qquad
\includegraphics[width=3.5cm]{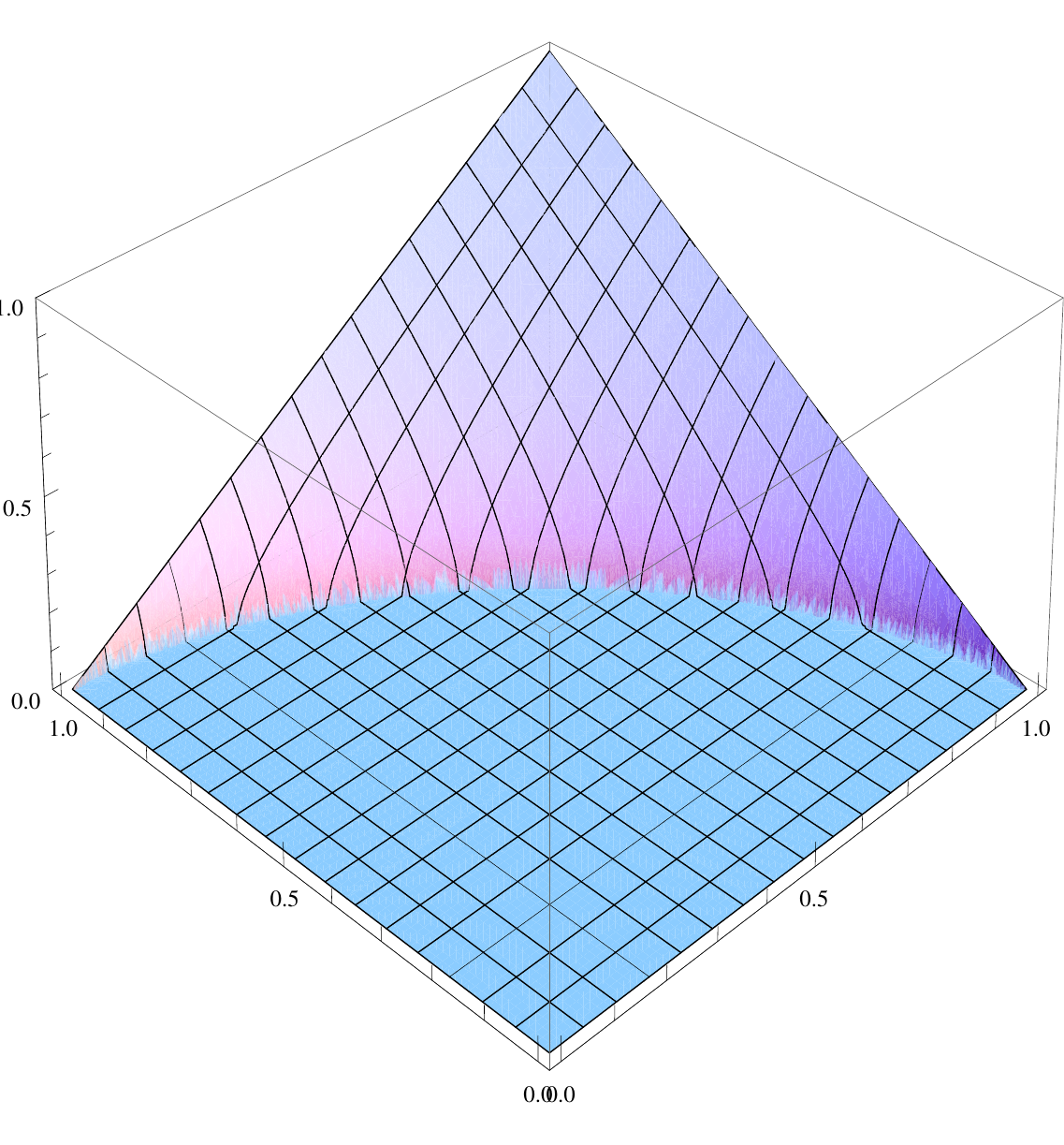}
\caption{The DP t-norm, the RDP $t$-norm $*_{2/3}$ and the Schweizer-Sklar $t$-norm $*_{SS}^{\log_{3/2} 2}$.}
\end{center}
\vspace{-2em}
\end{figure}

Moreover, $t$-norms isomorphic with the {\L}ukasiewicz one are continuous functions over $[0,1]^2$, while RDP $t$-norms are only left-continuous.

On the other hand, if we do not require a MTL logic to be standard complete, that is, complete with respect to a set of {\em standard} algebras
(algebras $([0,1],*,\Rightarrow,\min,\max,0,1)$, where $*$ is a $t$-norm, and $\Rightarrow$ its associated residuum),
we can naturally study the logic of residuated drastic product chains, whose class, clearly, does not contain any standard algebra.
Needless to say, the drastic product chains defined on subsets of $[0,1]$ coincide with the drastic product $t$-norm over their whole
universes.
Further, the logic of these chains is nicely axiomatised by a slightly weakened version of the law of the excluded middle.

It turns out that the logic of all residuated drastic product chains is the logic called S$_3$MTL in \cite{nog,hnp},
where some of its properties are stated and proved.

In this note we shall justify the study of S$_3$MTL, that we rename DP for Drastic Product logic, in the light
of several interesting logico-algebraic properties.
Further, we introduce a category dually equivalent to finite DP-algebras and utilise it to classify all schematic
extensions of DP, and to characterise the finitely generated free DP-algebras.

\section{Preliminaries}

We assume that the reader is acquainted with many-valued logics in Hájek's sense, and with their algebraic semantics.
We refer to \cite{haj,hand} for any unexplained notion.
We recall that MTL is the logic, on the language $\{\&,\wedge,\to,\bot\}$, of all left-continuous $t$-norms and their residua, and that its associated algebraic
semantics in the sense of Blok-Pigozzi \cite{bp} is the variety $\mathbb{MTL}$ of MTL-algebras $(A,*,\Rightarrow,\sqcap,\sqcup,0,1)$, that is, prelinear, commutative, bounded, integral, residuated lattices \cite{hand}. Derived connectives are negation $\neg \varphi := \varphi \to \bot$, top element $\top := \neg \bot$,
lattice disjunction $\varphi \vee \psi := ((\varphi \to \psi) \to \psi) \wedge ((\psi \to \varphi) \to \varphi)$.
The connectives $\&$,$\wedge$,$\vee$ are modeled by the monoidal and lattices operations $*$,$\sqcap$,$\sqcup$, while $\to$ by the residuum $\Rightarrow$ and $\bot,\top$ by the elements $0$,$1$. On the algebraic side: $\sim x \df x \Rightarrow 0$.

Every axiomatic extension L of MTL has its associate algebraic semantics: a subvariety $\mathbb{L}$ of $\mathbb{MTL}$ such that a formula $\varphi$ is a theorem of L iff the equation $\varphi = \top$
holds in any algebra of $\mathbb{L}$. $\mathbb{BL}$ is axiomatized as $\mathbb{MTL}$ plus $\varphi\land\psi=\varphi\&(\varphi\to\psi)$, $\mathbb{MV}$ as $\mathbb{BL}$ plus $\neg\neg\varphi=\varphi$, and $\mathbb{MV}_3$ as $\mathbb{MV}$ plus $\varphi\&\varphi=(\varphi\&\varphi)\&\varphi$.

\noindent{\em Drastic product chains} are MTL-chains $(A,*,\Rightarrow,\sqcap,\sqcup,0,1)$ s.t., for all $x,y \in A$,
\begin{equation}\label{eq:dp}
x*y\df\begin{cases}
0&\text{if }x,y<1,\\
\min\{x,y\}&\text{otherwise}.	
\end{cases}
\end{equation}
We denote by $\mathbb{DP}$ the subvariety of $\mathbb{MTL}$ generated by all drastic product chains. The members of $\mathbb{DP}$ are called \emph{drastic product algebras}.

Noguera points out in \cite[Page 108]{nog} that $\mathbb{DP}$ coincides with the variety named $\mathbb{S}_{3}\mathbb{MTL}$.
Each logic in the hierarchy S$_k$MTL (for $2 \leq k \in \mathbb{Z}$) is axiomatised by a generalised form of the {\em excluded middle} law:
$\varphi \vee \neg(\varphi^{k-1})$ (where $\varphi^0 := \top$ and $\varphi^n := \varphi^{n-1} \& \varphi$), whence the logic of drastic product DP is axiomatised
by
\begin{equation*}\label{eq:DP}
\tag*{(DP)}\varphi\vee\neg(\varphi^{2}).
\end{equation*}
Clearly, the logic S$_2$MTL, axiomatised as MTL plus $\varphi \vee \neg \varphi$ is just classical Boolean logic,
as the latter axiom is {\em the} excluded middle law.

The basic example of a drastic product chain is, for any real $c$, with $0 < c < 1$, the algebra
$[0,1]^c := ([0,c] \cup \{1\},*,\Rightarrow,\sqcap,\sqcup,0,1)$ where $*$ is defined as in (\ref{eq:dp}), while
\begin{equation}\label{eq:implication}
x\Rightarrow y\df\begin{cases}
1&\text{if }x\leq y,\\
c&\text{if }1>x>y,\\
y&\text{if }x=1.
\end{cases}
\end{equation}

Equations (\ref{eq:dp}) and (\ref{eq:implication}) express the operations of any DP-chain. Indeed:

\begin{lemma}\label{lemma:coatom}
A non-trivial MTL-chain $\mathcal{A} = (A,*,\Rightarrow,\sqcap,\sqcup,0,1)$ is a DP-chain iff it has a coatom $c$, and $x*x = 0$
for all $1 > x \in A$.
If $\mathcal{A}$ is a DP-chain then $*$ and $\Rightarrow$ are defined as in (\ref{eq:dp}) and (\ref{eq:implication}).
Moreover,
if $c > 0$ then $c = \sim c$ is its only negation fixpoint.% of $A$.
\end{lemma}
\begin{proof}
Assume $\mathcal{A}$ is a DP-chain: if $\mathcal{A} \cong \{0,1\}$ the claim trivially holds.
Assume then $\vert A\vert>2$. Clearly, $x * x = 0$ for all $1 > x \in A$.
Since $*$ is non-decreasing, $x*y = 0$ for all $x,y < 1$.
Take $y < x < 1 \in A$, and let $c = x \Rightarrow y$. By the properties
of residuum, $c < 1$. Take now any $z < 1$ in $A$.
Since $x * z = 0$ we have $z \leq c$. Hence $c$ is the coatom of $A$.
It is now easy to check that (\ref{eq:dp}) and (\ref{eq:implication})
define $*$ and $\Rightarrow$.
For the other direction notice that if $\mathcal{A}$ is an MTL-chain satisfying
the two assumptions, then $x*x \in \{0,1\}$ for all $x \in A$, and hence
$\mathcal{A}$ satisfies \ref{eq:DP}.
The lemma follows noting that if $c > 0$ then $c * 1 = c$ and $c * c = 0$, hence
$c = \sim c$ (an MTL-chain may have at most one negation fixpoint).
\end{proof}

\begin{remark}\label{rem:standardfail}
Lemma \ref{lemma:coatom} shows that $\mathbb{DP}$ does not contain any
{\em standard algebra}, that is a chain whose lattice reduct is $([0,1],\leq_{\mathbb{R}})$
(where $\leq_{\mathbb{R}}$ denotes the restriction of the standard order of real numbers).
It must be stressed, however, that for any $0 < c < 1$, the operation $*$ of the algebra $[0,1]^c$
coincides with the drastic product $t$-norm $*_D$ wherever defined.
\end{remark}
\section{DP, RDP and WNM}
We recall that $\mathbb{WNM}$ and $\mathbb{RDP}$ are the subvarieties of $\mathbb{MTL}$ respectively
satisfying the identities \ref{eq:wnm}, and both \ref{eq:wnm} and \ref{eq:rdp}, given below.
\begin{align*}
\tag*{(wnm)}&\neg(\varphi\&\psi)\vee((\varphi\land \psi)\to(\varphi\&\psi))\,=\,\top\,.\label{eq:wnm}\\
\tag*{(rdp)}&(\varphi\to \neg \varphi)\vee\neg\neg \varphi\,=\,\top\,.\label{eq:rdp}
\end{align*}
\begin{proposition}\label{rrdp}
$\mathbb{DP}\subset\mathbb{RDP}\subset\mathbb{WNM}$.
\end{proposition}
\begin{proof}
Lemma \ref{lemma:coatom} shows immediately that DP-chains satisfy both identities. %\ref{eq:rdp} and \ref{eq:wnm}.
\end{proof}

We recall that a variety $\mathbb{V}$ is {\em locally finite}
whenever every finitely generated subalgebra of an algebra in $\mathbb{V}$ is finite.
Equivalently, free finitely generated algebras are finite.
In a locally finite variety the three classes of finitely generated, finitely presented, and finite algebras coincide.
Now, since $\mathbb{WNM}$ is locally finite (see \cite[Proposition 9.15]{nog}), from \Cref{rrdp} we obtain:

\begin{corollary}\label{teo:cf}
$\mathbb{DP}$ is locally finite and is generated by the class of all finite DP-chains.
\end{corollary}
\begin{proposition}\label{lemma:intersections}
$\mathbb{DP} \cap \mathbb{BL} = \mathbb{NM} \cap \mathbb{BL} = \mathbb{MV}_3$.
$\mathbb{RDP} \cap \mathbb{BL} = \mathbb{WNM} \cap \mathbb{BL} = \mathbb{MV}_3 \oplus \mathbb{G}$,
where $\mathbb{NM}$ (nilpotent minimum) is $\mathbb{WNM}$ plus $\neg\neg\varphi=\varphi$, and
$\mathbb{MV}_3 \oplus \mathbb{G}$ is the variety generated by the ordinal sum of the $3$-element MV-chain with the standard G\"{o}del algebra.
\end{proposition}
\begin{proof}
The first two equalities are shown in \cite{nog,hnp}. For the other two equalities, a direct inspection shows that a BL-chain satisfies \ref{eq:wnm} iff it is isomorphic to an ordinal sum whose first component is an MV-chain with no more than three elements, and the others (if present) are isomorphic to $\{0,1\}$. Finally, note that \ref{eq:rdp} holds in a BL-chain iff it is isomorphic to an ordinal sum whose first component is an MV-chain with no more than three elements. So, if a BL-chain models \ref{eq:wnm}, then it satisfies also \ref{eq:rdp}.
\end{proof}
\section{Canonical completeness}

We have seen that DP is axiomatised from MTL by a weakened form of excluded middle law, and that algebras $[0,1]^c$
are, defensibly, good approximators of the drastic product $t$-norm, in the sense of \cite{ssac,rdp}.
In this section we show that each algebra $[0,1]^c$ is a canonical model of the logic DP.

Recall that an extension L of MTL is {\em standard complete} if $\mathbb{L}$ is generated
by a set of standard algebras (see Remark \ref{rem:standardfail}).

Notice that DP is not standard complete since Lemma \ref{lemma:coatom} shows that each DP-chain must have a coatom. However,
it must be noticed that the form of completeness DP enjoys is precisely the same
that is enjoyed by classical propositional logic, which technically is not a standard
complete logic either.
To stress this fact we propose here the following notion of completeness,
which strengthens the notion of single-chain completeness (an extension L of MTL is {\em single-chain complete} whenever its associated variety $\mathbb{L}$ is generated by a chain, \cite{ssc1}).
\begin{definition}\label{def:canonical}
{\rm
A schematic extension L of MTL is {\em canonically complete} if it is complete with
respect to a single algebra $\mathcal{A}$, called {\em canonical model} of L, such that:
\begin{itemize}
\item The lattice reduct of $\mathcal{A}$ is a sublattice of $\lag [0,1],\leq_\mathbb{R}\rog$.
\item For every L-chain $\mathcal{B}$ whose lattice reduct is a sublattice of $\lag [0,1],\leq_\mathbb{R}\rog$, there is $\mathcal{A}'\cong \mathcal{A}$ such that $\lag B,\leq_B\rog$ is a sublattice of $\lag A',\leq_{A'}\rog$.
\end{itemize}
}
\end{definition}
In other terms, $\mathcal{A}$ generates $\mathbb{L}$ and is (up to isomorphism of MTL-algebras) lattice-inclusion-maximal among the algebras in $\mathbb{L}$ whose lattice reduct is a sublattice of
$\lag [0,1],\leq_\mathbb{R}\rog$.
With this definition in place, we list some examples:
\begin{itemize}
\item Classical propositional logic is canonically complete even though
not standard complete;
\item G\"{o}del, product and {\L}ukasiewicz logic are both canonically and standard complete,
as it is BL (w.r.t. the ordinal sum of $\omega$ copies of the standard MV-algebra, for instance) \cite{hand,haj,blvar}.
\item MTL is standard complete, but it is not known whether it is canonically complete, nor single chain complete (\cite{ssc1}). The same
applies to IMTL.
\item The logic WCBL (\cite[Ch. 7.2,7.3]{nog}) obtained extending BL with the weak cancellativity axiom
$\neg(\varphi\&\psi)\vee ((\varphi\to(\varphi\&\psi))\to\psi)$ is standard complete, being complete w.r.t. the set formed by the standard MV-algebra
and the standard product algebra,
but it is not canonically complete. Indeed, an MTL-chain belongs to this variety iff it is a product or an MV-chain: hence WCBL is not single chain complete.
\item Each logic BL$^n$ (with $n \geq 2$), axiomatised as BL plus the $n$-contraction axiom $\varphi^n\to\varphi^{n+1}$ (see \cite{bln}),
is neither standard nor canonically complete.
In the associated variety the only standard algebra is the standard G\"{o}del algebra, but there are no generic chains at all.
\end{itemize}
\begin{lemma}\label{stdiso}
All the algebras of the form $[0,1]^c$ are isomorphic.
\end{lemma}
\begin{proof}
By Lemma \ref{lemma:coatom}.
\end{proof}
\begin{theorem}\label{thm:infinitechain}
$\mathbb{DP}$ is generated by any infinite DP-chain.
\end{theorem}
\begin{proof}
By \Cref{teo:cf}, we have that if an equation fails in some DP-algebra, then it fails in some finite DP-chain.
Take now an infinite DP-chain $\mathcal{A}$. Denote by $c$ its coatom.
By Lemma \ref{lemma:coatom}, we have that every finite DP-chain $\mathcal{B}$ embeds into $\mathcal{A}$:
Trivially, $\{0,1\}\hookrightarrow\mathcal{A}$.
If $\vert\mathcal{B}\vert>2$, call $b$ its coatom.
Then every injective order preserving mapping $\phi$ from $\mathcal{B}$ to $\mathcal{A}$ such that $\phi(0)=0$, $\phi(1)=1$, and $\phi(b)=c$ (note that such a map always exists, since $\mathcal{B}$ is finite) is such that $\phi \colon \mathcal{B} \hookrightarrow \mathcal{A}$.
Hence every equation that fails in some finite DP-chain also fails in $\mathcal{A}$.
\end{proof}
Theorem \ref{thm:infinitechain} together with \Cref{teo:cf} proves the following result:
\begin{theorem}\label{thm:unboundedset}
$\mathbb{DP}$ is generated by any set of DP-chains of unbounded cardinality.
\end{theorem}

\begin{theorem}\label{thm:DPcanonical}
The logic DP enjoys the strong completeness w.r.t. $[0,1]^c$, with $c \in (0,1)$.
Moreover, DP is not standard complete but it is canonically complete.
\end{theorem}
\begin{proof}
By \cite[Theorem 3.5]{dist} it is enough to show that every countable DP-chain embeds into $[0,1]^c$. Let $\mathcal{B}$ be a countable DP-chain.
Reasoning as in the proof of \Cref{thm:infinitechain} we find the desired embedding $\phi \colon \mathcal{B} \hookrightarrow [0,1]^c$ (preserving $0,1$ and the coatom). Such $\phi$ exists because $[0,c]$ is an uncountable dense linear order.
The latter statement follows from Remark \ref{rem:standardfail} and Lemma \ref{stdiso}.
\end{proof}

\section{S$_n$MTL and the definability of the $\Delta$ operator}

DP coincides with S$_3$MTL. In this section we shall recall some interesting properties
of S$_n$MTL-algebras from \cite{nog,hnp,kow}, and relate them to the definability
of the $\Delta$ projection operator \cite{bdelta}.

We recall that a variety is {\em semisimple} if all its subdirectly irreducible algebras are simple,
and it is a {\em discriminator} variety if the ternary discriminator $t$
($t(x,x,z) = z$, while $t(x,y,z) = x$  if $x \neq y$), is definable on every subdirectly irreducible algebra.
\begin{theorem}[\cite{nog,hnp,kow}]\label{teo:sn}
Let $\mathbb{L}$ be a variety of MTL-algebras. Then the following are equivalent:
\begin{itemize}
\item $\mathbb{L}$ is semisimple.
\item $\mathbb{L}$ is a discriminator variety.
\item $\mathbb{L}$ is a subvariety of $\mathbb{S}_n\mathbb{MTL}$ for some $n \geq 2$.
\item Every chain in $\mathbb{L}$ is simple and n-contractive (i.e. it satisfies $x^n=x^{n-1}$), for some $n\geq 2$.
\end{itemize}
\end{theorem}
Given a schematic extension L of MTL we write L$_\Delta$
%variety $\mathbb{V}$ of MTL-algebras we write $\mathbb{V}_\Delta$
for the extension/expansion of L
%$\mathbb{V}$
with the $\Delta$ unary projection connective,
%projection operator,
axiomatised as follows:
\begin{alignat*}{2}
&\Delta \varphi \vee \neg\Delta \varphi, \qquad\Delta \varphi \to \varphi, && \quad\Delta\varphi\to \Delta \Delta \varphi,\\
&\Delta(\varphi\to \psi)\to (\Delta\varphi \to \Delta\psi),&&\quad\Delta(\varphi \vee \psi)\to (\Delta \varphi \vee \Delta \psi).
%&\Delta(x) \sqcup \sim\Delta(x)=1.\quad\Delta(x)\Rightarrow x=1.&&\quad\Delta(x)\Rightarrow \Delta(\Delta(x))=1.\\
%&\Delta(x\Rightarrow y)\Rightarrow (\Delta(x)\Rightarrow \Delta(y))=1.&&\quad\Delta(x \sqcup y)\Rightarrow (\Delta(x) \sqcup \Delta(y))=1.
\end{alignat*}
Recall that on every MTL-chain $A$ and every $x \in A$,
the identities associated with these axioms model the operation %defined by
%it holds that
$\Delta x = 1$ if $x = 1$, while $\Delta x = 0$ if $x < 1$.
%(we use the same symbol for connective and operation).
\begin{proposition}\label{prop:delta}
Let $\mathbb{L}$ be a variety of MTL-algebras.
Then $\mathbb{L}_\Delta$ is semisimple.
\end{proposition}
\begin{proof}
We prove that each chain in $\mathbb{L}_\Delta$ is simple.

Take a chain $\mathcal{A}\in\mathbb{L}_\Delta$. For every non-trivial congruence $\theta$ on $\mathcal{A}$, it holds that $\lag a,b\rog\in\theta$, for some $a\neq b$. Then exactly one of $a\Rightarrow b$ and $b\Rightarrow a$ is $1$, say $a\Rightarrow b$. Hence $\lag b\Rightarrow a, 1\rog\in\theta$, and $\lag\Delta(b\Rightarrow a),\Delta(1)\rog\in\theta$. That is, $\lag 0, 1\rog\in\theta$.
Since $\theta$ is a congruence of the lattice reduct of $\mathcal{A}$, all elements between $0$ and $1$ are in the same class, which means
that $\mathcal{A}$ is simple.
\end{proof}
\begin{theorem}\label{thm:delta}
Let $\mathbb{L}$ be a variety of MTL-algebras.
Then $\mathbb{L}_\Delta = \mathbb{L}$ iff $\mathbb{L}$ is a subvariety of $\mathbb{S}_k\mathbb{MTL}$ for some
integer $k > 1$.
\end{theorem}
\begin{proof}
It is immediate to check that the $\Delta$ operator is definable in each variety $\mathbb{S}_k\mathbb{MTL}$
as $\Delta x = x^{k-1}$.

For the other direction, assume $\Delta$ is definable in a variety $\mathbb{L}$ of MTL-algebras.
Then $\mathbb{L}$ is semisimple by \Cref{prop:delta}.
By \Cref{teo:sn}, $\mathbb{L}$ is a subvariety of $\mathbb{S}_k\mathbb{MTL}$ for some
integer $k > 1$.
\end{proof}

\begin{remark}
Theorem \ref{thm:delta} shows that if $\Delta$ is definable in some subvariety $\mathbb{L}$ of MTL-algebras,
then it is always definable as $\Delta x = x^k$ for some $k \geq 1$. Further,
$\mathbb{L}$ is a discriminator variety iff it defines $\Delta$ and, in this case, $(\Delta(x \leftrightarrow y) \wedge z) \vee (\neg\Delta(x \leftrightarrow y) \wedge x)$
is the discriminator term.

Notice that the assumptions of \Cref{teo:sn} cannot be generalised to extensions/expansions of MTL.
For instance, consider the logic G$_\backsim$ introduced in \cite{eghn}, which is G\"{o}del
logic extended/expanded with an independent involutive negation $\backsim$.
It is an exercise to check that $\Delta$ is definable in G$_\backsim$ as $\Delta x = \neg \backsim x$,
but not as $x^k$ for any integer $k$. Hence $\mathbb{G}_\backsim$ is a discriminator variety, but G$_\backsim$ is not
an extension/expansion of any S$_k$MTL.
\end{remark}
\begin{corollary}\label{cor:deltainBL}
A variety $\mathbb{L}$ of BL-algebras coincides with $\mathbb{L}_\Delta$ iff
it is a variety of MV-algebras generated by a finite set of finite MV-chains.
\end{corollary}
\begin{proof}
By \cite[Corollary 8.16]{nog},
$\mathbb{S}_n\mathbb{MTL} \cap \mathbb{BL}$
is generated by the set of all MV-chains with at most $n$ elements, for every integer $n>1$.
The result follows from \Cref{thm:delta}.
\end{proof}
The following result is needed in the next section.
\begin{lemma}\label{lemma:simpleRDP}
The classes of simple WNM-chains and of simple RDP-chains both coincide with the class of DP-chains.
\end{lemma}
\begin{proof}
The result follows from Proposition \ref{rrdp} and from \Cref{teo:sn}, since every WNM-chain satisfies $x^3=x^{2}$.
\end{proof}
\section{A dual equivalence}
In \cite{bval} a dual equivalence between the category of finite RDP-algebras and homomorphisms
and the category $\mathsf{HF}$ of {\em finite hall forests} is proven.
We recall here that a finite hall forest is a finite multiset whose elements are pairs $(T,J)$,
where $T$ is a finite {\em tree} (that is, a poset with minimum such that the downset of each element is a chain)
and $J$ is a (possibily empty) finite chain, while a morphism $h \colon \{(T_i,J_i)_{i \in I}\}  \to \{(T_k,J_k)_{k \in K}\}$
is a family of pairs $\{(f_i,g_i)\}_{i \in I}$ such that for each $i \in I$ there is $k \in K$ such that $f_i \colon T_i \to T_k$,
and $g_i \colon J_i \to J_k$ are order-preserving downset preserving
maps, with the additional constraint that $g_i(\max J_i) = \max J_k$. In case $J_k$ is empty it is stipulated that $g_i$ is
the partial, nowhere defined, map.

For each integer $k > 0$ let ${\mathbf k}$ denote the $k$-element chain.
\begin{definition}\label{def:multiofchains}
{\rm
Let $\mathsf{MC}$ be the category whose objects are finite multisets of (nonempty) finite chains,
and whose morphisms $h \colon C \to D$, are defined as follows. Display $C$ as
$\{C_1,\ldots,C_m\}$ and $D$ as $\{D_1,\ldots,D_n\}$. Then $h = \{h_i\}_{i = 1}^m$,
where each $h_i$ is an order preserving surjection $h_i \colon C_i \twoheadrightarrow D_j$ for some $j = 1,2,\ldots,n$.
Let $\mathsf{MC}^\top$ be the nonfull subcategory of $\mathsf{MC}$ whose morphisms $h \colon C \to D$
satisfy the following additional constraint: for each $i = 1,2,\ldots,m$, if the target $D_j$ of $h_i$ is not isomorphic with $\mathbf{1}$,
then $h_i^{-1}(\max D_j) = \{\max C_i\}$.
}
\end{definition}
\begin{theorem}\label{thm:duality}
The category $\mathsf{MC}^\top$ is equivalent to the full subcategory of $\mathsf{HF}$ whose objects have the form
$\{({\mathbf 1},J_i)\}_{i \in I}$.
Whence, $\mathsf{MC}^\top$ is dually equivalent to the category $\mathbb{DP}_{fin}$ of finite DP-algebras and their homomorphisms.
\end{theorem}
\begin{proof}
Note that the only map from $\mathbf{1}$ to itself is the identity $id_{\mathbf{1}}$.
Then direct inspection shows that the functor $Tr \colon \mathsf{MC}^\top \to \mathsf{HF}$, defined on objects as
$Tr(\{C_1,\ldots,C_m\}) = \{(\mathbf{1},C_1 \setminus \{\max C_1\}),\ldots,(\mathbf{1},C_m \setminus \{\max C_m\})\}$, and on morphisms as
$Tr(\{h_i\}_{i = 1}^m) = \{(id_{\mathbf{1}},h_i \upharpoonright C_i \setminus \{\max C_i\})\}_{i=1}^m$
(we agree that $h_i \upharpoonright \emptyset$ is the nowhere defined map), implements the equivalence stated in the first statement.
Observe that the dual of a simple RDP-algebra in the category $\mathsf{HF}$ is a hall forest of the form $\{(\mathbf{1},J)\}$.
The last statement then follows from Lemma \ref{lemma:simpleRDP}.
\end{proof}

Clearly, the multiset $\{\mathbf{1}\}$ is the terminal object of $\mathsf{MC}^\top$.
%while $\emptyset$ is the initial one.
%Clearly,
%The multisets $\{\mathbf{1}\}$ and $\emptyset$ are resp. the terminal and the initial object of $\mathsf{MC}^\top$.
Applying the first equivalence of Theorem \ref{thm:duality}
one can verify the following constructions, as they are carried over to $\mathsf{MC}^\top$ from $\mathsf{HF}$. Given two objects $C,D \in \mathsf{MC}^\top$,
the coproduct object $C \uplus D$ of $C$ and $D$ is just the disjoint union of the multisets $C$ and $D$;
the product object $C \times D$ is computed using the following $\mathsf{MC}^\top$ isomorphisms.
First, products distribute over coproducts: $C \times (D \uplus E) \cong (C \times D) \uplus (C \times E)$.
Given $C \in \mathsf{MC}^\top$, let $C^\top$ denote the object obtained adding to each chain in $C$
a fresh maximum. Then $\{\mathbf{i}\} \times \{\mathbf{1}\} \cong \{\mathbf{i}\}$ and
$\{\mathbf{i + 1}\} \times \{\mathbf{2}\} \cong \{\mathbf{i + 1}\}$. Moreover,
$$
\{\mathbf{i+2}\} \times \{\mathbf{j+2}\} \cong
((\{\mathbf{i+2}\} \times \{\mathbf{j+1}\})
 \uplus (\{\mathbf{i+1}\} \times \{\mathbf{j+1}\}) \uplus (\{\mathbf{i+1}\} \times \{\mathbf{j+2}\}))^\top.
$$

Denote by $MC \colon {\mathbb {DP}}_{fin} \to {\mathsf {MC}}^\top$ the functor implementing the dual equivalence.
The following lemma is straightforward.
\begin{lemma}\label{lemma:dualchain}
For any integer $i > 0$,
$MC^{-1}\, \{\mathbf{i}\}$
is the DP-chain with $i+1$ elements.
\end{lemma}

Let $\mathbf{F}_k$ denote the free DP-algebra over a set of $k$ many free generators.
In the following we write $n C$ for the $n$th copower of $C \in \mathsf{MC}^\top$.
\begin{theorem}\label{thm:free}
$MC\,\mathbf{F}_1$ is the multiset of chains $\{\mathbf{1},\mathbf{3},\mathbf{2},\mathbf{1}\}$.
Hence, $\mathbf{F}_1$ has exactly $2^2 \cdot 3 \cdot 4 = 48$ elements. More generally,
\begin{equation}\label{eq:free}
MC\,\mathbf{F}_k \cong 2^k \{{\mathbf 1}\}\uplus (3^k - 2^k) \{{\mathbf 2}\} \uplus
\biguplus_{h=3}^{k+2}\left( \sum_{i=0}^{h-2} (-1)^i {h-2 \choose i}(h+1-i)^k \right)\{\mathbf{h}\}\,.
\end{equation}
Whence, the cardinality of $\mathbf{F}_k$ is given by
\begin{equation}\label{eq:cardfree}
2^{2^k} \cdot 3^{3^k-2^k} \cdot \prod_{h=3}^{k+2}(h+1)^{\sum_{i=0}^{h-2} (-1)^i {h-2 \choose i}(h+1-i)^k}\,.
\end{equation}
\end{theorem}
\begin{proof}
For what regards $\mathbf{F}_1$
the proof follows at once by Theorem \ref{thm:duality}, Lemma \ref{lemma:dualchain} and \cite[Proposition 6]{bval}.
For the general case, recall that $\mathbf{F}_k$ is the $k$th copower of $\mathbf{F}_1$, hence
$MC\,\mathbf{F}_k$ is given by the $k$th power of $\{\mathbf{1},\mathbf{3},\mathbf{2},\mathbf{1}\}$.
Proceeding by induction on $k$,
we denote by $a_i^{(k)}$ the coefficient multiplier of $\{\mathbf{i}\}$ in Eq. (\ref{eq:free}).
Using the fact that $\{\mathbf{k}\} \times \{\mathbf{3}\} \cong (k-1)\{\mathbf{k+1}\} \uplus (k-2)\{\mathbf{k}\}$,
the computation of $MC\,\mathbf{F}_k \times \{\mathbf{1},\mathbf{3},\mathbf{2},\mathbf{1}\}$ gives the
recurrences $a_1^{(k+1)} = 2 a_1^{(k)}$, $a_2^{(k+1)} = a_1^{(k)} + 3 a_2^{(k)}$, $a_3^{(k+1)} = a_1^{(k)} + a_2^{(k)} + 4 a_3^{(k)}$, and
$a_h^{(k+1)} = (h-2)a_{h-1}^{(k)} + (h+1)a_h^{(k)}$ for $h > 3$, whose solutions finally yield Eq. (\ref{eq:free}).
%A tedious computation by induction on $k$, using the fact that $\{\mathbf{k}\} \times \{\mathbf{3}\} \cong (k-1)\{\mathbf{k+1}\} \uplus (k-2)\{\mathbf{k}\}$,
%proves Eq. (\ref{eq:free}). As a matter of fact, denoting $a_i^{(k)}$ the multiplier coefficient of $\{\mathbf{i}\}$ in Eq. \ref{eq:free},
%we have the recurrences $a_1^{(k+1)} = 2 a_1^{(k)}$, $a_2^{(k+1)} = a_1^{(k)} + 3 a_2^{(k)}$, $a_3^{(k+1)} = a_1^{(k)} + a_2^{(k)} + 4 a_3^{(k)}$, and
%$a_h^{(k+1)} = (h-2)a_{h-1}^{(k)} + (h+1)a_h^{(k)}$ for $k > 3$, whose solutions finally yield Eq. \ref{eq:free}.
By Lemma \ref{lemma:dualchain},  Eq. (\ref{eq:cardfree}) yields the cardinality of $\mathbf{F}_k$.
\end{proof}
Notice that by replacing in  Eq. (\ref{eq:cardfree}) each base number with the DP-chain with the same cardinality
one gets the decomposition of $\mathbf{F}_k$ as direct product of chains.

By Theorem \ref{thm:unboundedset}, %each family of DP-chains of unbounded cardinality generates $\mathbb{DP}$.
%Hence,
any proper subvariety of $\mathbb{DP}$ is generated by a necessarily finite family of finite chains.
We can then classify all subvarieties of $\mathbb{DP}$.

\begin{theorem}\label{thm:subvarieties}
Each proper subvariety of $\mathbb{DP}$ is generated by a finite chain.
Moreover, two finite chains of different cardinality generate distinct subvarieties.
\end{theorem}
\begin{proof}
First note that each non-trivial finite DP-chain embeds into any DP-chain of greater cardinality.
Whence, each proper subvariety $\mathbb{V}$ of $\mathbb{DP}$ is generated by a finite chain,
which is the chain with maximum cardinality among any given set of chains generating $\mathbb{V}$.
For each object $C \in \mathsf{MC}^\top$ let its {\em height} $H(C)$ be the maximum cardinality
of its chains. It is easy to see that $H(C \uplus D) = \max\{H(C),H(D)\}$ and that
given maps $C_1 \hookrightarrow D_1$ and $C_2 \twoheadrightarrow D_2$, it holds that
$H(C_1) \leq H(D_1)$ and $H(D_2) \leq H(C_2)$. It follows by Thm. \ref{thm:duality}, Lemma \ref{lemma:dualchain} and the HSP theorem
that all finite algebras
in the variety generated by the $k$-element DP-chain must have dual of height $\leq k-1$.
Let $\mathbb{V}_h$ be the subvariety generated by the $h$-element chain.
Then, if $h < k$, $(\mathbb{V}_h)_{fin}$ is properly included in $(\mathbb{V}_k)_{fin}$.
By \cite[Cor. VI.2.2]{johnstone} and Cor. \ref{teo:cf}, $\mathbb{V}_h$ is properly included in $\mathbb{V}_k$.
\end{proof}

One may wonder whether there are classes of residuated lattices dually equivalent to $\mathsf {MC}$.
The answer is in the positive, for it is easy to prove that the class of finite algebras in $\mathbb{G}_\Delta$, the variety of
G\"{o}del algebras plus $\Delta$, is the category sought for. It turns out that $\mathbb{DP}$ is equivalent to a non-full subcategory
of $\mathbb{G}_\Delta$, with the same objects, but with fewer morphisms.

Even if we do not have analyzed the first-order case, there is a result that may have some interest.
The logic DP$\forall$ enjoys strong completeness w.r.t. $[0,1]^c$, with $c\in (0,1)$. This can be proved via a modification of a construction described in \cite{dist}, that allows to embed every countable DP-chain into a chain isomorphic to $[0,1]^c$ (for some $c\in (0,1)$), by preserving all $\inf$'s and $\sup$'s. By \cite[Theorems 5.9, 5.10]{dist} this suffices to prove DP$\forall$ is strongly complete w.r.t. $[0,1]^c$.
\bibliography{perfectwnm}

\newcommand{\etalchar}[1]{$^{#1}$}
\providecommand{\bysame}{\leavevmode\hbox to3em{\hrulefill}\thinspace}
\providecommand{\MR}{\relax\ifhmode\unskip\space\fi MR }
\providecommand{\MRhref}[2]{%
  \href{http://www.ams.org/mathscinet-getitem?mr=#1}{#2}
}
\providecommand{\href}[2]{#2}
\begin{thebibliography}{EGHN00}

\bibitem[{Baa}96]{bdelta}
M.~{Baaz}, \emph{{Infinite-valued G\"odel logics with 0-1-projections and
  relativizations}}, {G\"odel '96. Logical foundations of mathematics, computer
  science and physics -- Kurt G\"odel's legacy}, Berlin: Springer-Verlag, 1996,
  pp.~23--33.

\bibitem[BM11]{bln}
M.~Bianchi and F.~Montagna, \emph{{$n$-contractive BL-logics}}, Arch. Math.
  Log. \textbf{50} (2011), no.~3-4, 257--285,
  \href{http://dx.doi.org/10.1007/s00153-010-0213-8}{doi:10.1007/s00153-010-0213-8}.

\bibitem[BP89]{bp}
W.~Blok and D.~Pigozzi, \emph{Algebraizable logics}, vol.~77, Memoirs of The
  American Mathematical Society, no. 396, American Mathematical Society, 1989,
  Available on \url{http://ow.ly/rV2GS}.

\bibitem[BV12]{bval}
S.~{B}ova and D.~Valota, \emph{{Finite RDP-algebras: duality, coproducts and
  logic}}, J. Log. Comp. \textbf{22} (2012), no.~3, 417--450,
  \href{http://dx.doi.org/10.1093/logcom/exr006}{doi:10.1093/logcom/exr006}.

\bibitem[CEG{\etalchar{+}}09]{dist}
P.~Cintula, F.~Esteva, J.~Gispert, L.~Godo, F.~Montagna, and C.~Noguera,
  \emph{Distinguished algebraic semantics for t-norm based fuzzy logics:
  methods and algebraic equivalencies}, Ann. Pure Appl. Log. \textbf{160}
  (2009), no.~1, 53--81,
  \href{http://dx.doi.org/10.1016/j.apal.2009.01.012}{doi:10.1016/j.apal.2009.01.012}.

\bibitem[CHN11]{hand}
P.~Cintula, P.~H{\'a}jek, and C.~Noguera, \emph{{Handbook of Mathematical Fuzzy
  Logic}}, vol. 1 and 2, College Publications, 2011.

\bibitem[EG01]{eg}
F.~Esteva and L.~Godo, \emph{{Monoidal t-norm based logic: Towards a logic for
  left-continuous t-norms}}, Fuzzy sets Syst. \textbf{124} (2001), no.~3,
  271--288,
  \href{http://dx.doi.org/10.1016/S0165-0114(01)00098-7}{doi:10.1016/S0165-0114(01)00098-7}.

\bibitem[EGHN00]{eghn}
F.~Esteva, L.~Godo, P.~H{\'{a}}jek, and M.~Navara, \emph{Residuated fuzzy
  logics with an involutive negation}, Arch. Math. Log. \textbf{39} (2000),
  no.~2, 103--124,
  \href{http://dx.doi.org/10.1007/s001530050006}{doi:10.1007/s001530050006}.

\bibitem[Háj98]{haj}
P.~Hájek, \emph{Metamathematics of fuzzy logic}, paperback ed., Trends in
  Logic, vol.~4, Kluwer Academic Publishers, 1998,
  \href{http://www.springer.com/philosophy/logic/book/978-1-4020-0370-7}{ISBN:9781402003707}.

\bibitem[HNP07]{hnp}
R.~Hor\v{c}\'{i}k, C.~Noguera, and M.~Petr\'ik, \emph{{On $n$-contractive fuzzy
  logics}}, Math. Log. Q. \textbf{53} (2007), no.~3, 268--288,
  \href{http://dx.doi.org/10.1002/malq.200610044}{doi:10.1002/malq.200610044}.

\bibitem[Jen02]{jen}
S.~Jenei, \emph{{A note on the ordinal sum theorem and its consequence for the
  construction of triangular norms}}, Fuzzy Sets Syst. \textbf{126} (2002),
  no.~2, 199--205,
  \href{http://dx.doi.org/10.1016/S0165-0114(01)00040-9}{10.1016/S0165-0114(01)00040-9}.

\bibitem[Joh82]{johnstone}
P.T. Johnstone, \emph{Stone spaces}, Cambridge Studies in Advanced Mathematics,
  Cambridge University Press, 1982.

\bibitem[KMP00]{kmp}
E.P. Klement, R.~Mesiar, and E.~Pap, \emph{Triangular norms}, hardcover ed.,
  Trends in Logic, vol.~8, Kluwer Academic Publishers, 2000,
  \href{http://www.springer.com/philosophy/logic/book/978-0-7923-6416-0}{ISBN:978-0-7923-6416-0}.

\bibitem[Kow04]{kow}
T.~Kowalski, \emph{{Semisimplicity, EDPC and Discriminator Varieties of
  Residuated Lattices}}, Stud. Log. \textbf{77} (2004), no.~2, 255--265,
  \href{http://dx.doi.org/10.1023/B:STUD.0000037129.58589.0c}{doi:10.1023/B:STUD.0000037129.58589.0c}.

\bibitem[Mon05]{blvar}
F.~Montagna, \emph{{Generating the variety of BL-algebras}}, {Soft Comput.}
  \textbf{9} (2005), no.~12, 869--874,
  \href{http://dx.doi.org/10.1007/s00500-004-0450-z}{doi:10.1007/s00500-004-0450-z}.

\bibitem[Mon11]{ssc1}
F.~Montagna, \emph{Completeness with respect to a chain and universal models in
  fuzzy logic}, Arch. Math. Log. \textbf{50} (2011), no.~1-2, 161--183,
  \href{http://dx.doi.org/10.1007/s00153-010-0207-6}{doi:10.1007/s00153-010-0207-6}.

\bibitem[Nog06]{nog}
C.~Noguera, \emph{{Algebraic study of axiomatic extensions of triangular norm
  based fuzzy logics}}, Ph.D. thesis, IIIA-CSIC, 2006, Available on
  \url{http://ow.ly/rV2sL}.

\bibitem[SS63]{ssac}
B.~Schweizer and A.~Sklar, \emph{Associative functions and abstract
  semigroups}, Publ. Math. Debrecen \textbf{10} (1963), 69--81.

\bibitem[Wan07]{rdp}
S.~Wang, \emph{{A fuzzy logic for the revised drastic product \emph{t}-norm }},
  Soft Comput. \textbf{11} (2007), no.~6, 585--590,
  \href{http://dx.doi.org/10.1007/s00500-005-0024-8}{doi:10.1007/s00500-005-0024-8}.

\end{thebibliography}
\bibliographystyle{amsalpha}
\end{document}